\documentclass[12pt,a4paper,reqno]{amsart}
\usepackage{fullpage}
\usepackage{hyperref}
\usepackage{pslatex}

\newcommand{\odip}[2]{o _{#1}\!\left(#2\right)\mathchoice{\!}{}{}{}}
\newcommand{\odi}[1]{\odip{}{#1}}
\newcommand{\Odip}[2]{\mathcal{O}_{\,#1}\left(#2\right)}
\newcommand{\Odipm}[2]{\mathcal{O}_{#1} (#2)}
\newcommand{\Odi}[1]{\mathcal{O}\left(#1\right)}
\newcommand{\dx}{\mathrm{d}}
\newcommand{\e}{\mathrm{e}}
\newcommand{\ii}{\mathrm{i}}
\newcommand{\E}{\widetilde{\mathcal{E}}}
\newcommand{\C}{\mathcal{C}}
\newcommand{\eps}{\varepsilon}
\newcommand{\K}{\mathbf{k}}
\newcommand{\Stilde}{\widetilde{S}}
\newcommand{\gA}{\mathfrak{A}}
\newcommand{\gB}{\mathfrak{B}}
\newcommand{\gC}{\mathfrak{C}}

\newtheorem{Theorem}{Theorem}[section]
\newtheorem{Lemma}{Lemma}[section]

\allowdisplaybreaks

\title{On an average ternary problem with prime powers}

\author[M.~Cantarini, A.~Gambini, A.~Languasco, A.~Zaccagnini]
       {Marco Cantarini, Alessandro Gambini, \\
        Alessandro Languasco, Alessandro Zaccagnini}

\date{\today}

\subjclass[2010]{Primary 11P32. Secondary 11P55, 11P05}
\keywords{Waring-Goldbach problem; Hardy-Littlewood method}

\begin{document}

\begin{abstract}
We continue our work on averages for ternary additive problems with
powers of prime numbers in \cite{LanguascoZ2016a},
\cite{LanguascoZ2017e}, and \cite{CantariniGZ2018}.
\end{abstract}

\maketitle

\section{Introduction}

The problem of representing a large integer $n$, satisfying suitable
congruence conditions, as a sum of a prescribed number of powers
of primes, say $n = p_1^{k_1} + \cdots + p_s^{k_s}$, is classical.
Here $k_1$, \dots, $k_s$ denote fixed positive integers.
This class of problems includes both the binary and ternary Goldbach
problem, and Hua's problem.
If the \emph{density} $\rho = k_1^{-1} + \cdots + k_s^{-1}$ is large
and $s \ge 3$, it is often possible to give an asymptotic formula for
the number of different representations the integer $n$ has.
When the density $\rho$ is comparatively small, the individual problem
is usually intractable and it is reasonable to turn to the easier task
of studying the average number of representations, if possible
considering only integers $n$ belonging to a \emph{short} interval
$[N, N + H]$, say, where $H \ge 1$ is ``small.''

Here we study ternary problems: let $\K = (k_1, k_2, k_3)$ where
$k_1$, $k_2$ and $k_3$ are integers with $2 \le k_1 \le k_2 \le k_3$.
Our goal is to compute the average number of representations of a
positive integer $n$ as $p_1^{k_1} + p_2^{k_2} + p_3^{k_3}$, where
$p_1$, $p_2$ and $p_3$ are prime numbers (or powers of primes).
Let
\begin{equation}
\label{Rk-def}
  R(n; \K)
  =
  \sum_{n = m_1^{k_1} + m_2^{k_2} + m_3^{k_3}}
    \Lambda(m_1) \Lambda(m_2) \Lambda(m_3),
\end{equation}
where $\Lambda$ is the von Mangoldt function, that is,
$\Lambda(p^m) = \log(p)$ if $p$ is a prime number and $m$ is a
positive integer, and $\Lambda(n) = 0$ for all other integers.
For brevity, we write $\rho = k_1^{-1} + k_2^{-1} + k_3^{-1}$ for
the density of this problem.
It will also shorten our formulae somewhat to write
$\gamma_k = \Gamma(1 + 1 / k)$ for any real $k > 0$, where $\Gamma$ is
the Euler Gamma-function.

\begin{Theorem}
\label{ternary}
Let $\K = (k_1, k_2, k_3)$ where $2 \le k_1 \le k_2 \le k_3$ be a
triple of integers.
For every $\eps > 0$ there exists a constant $C = C(\eps) > 0$,
independent of $\K$, such that
\[
  \sum_{n = N + 1}^{N + H} R(n; \K)
  =
  \frac{\gamma_{k_1} \gamma_{k_2} \gamma_{k_3}}{\Gamma(\rho)}
  H N^{\rho - 1}
  +
  \Odip{\K}{H N^{\rho - 1}
    \exp\Bigl\{ -C \Bigl(\frac{\log N}{\log\log N}\Bigr)^{1/3} \Bigr\}}
\]
as $N \to +\infty$, uniformly for
$N^{1 - 5 / (6 k_3) + \eps} < H < N^{1 - \eps}$.
\end{Theorem}

We recall the results in \cite{LanguascoZ2016a}, which correspond to
$\K = (1, 2, 2)$: here we must have $k_1 \ge 2$ because of the
limitation in the key Lemma~\ref{Tolev-Lemma}.
Theorem~\ref{ternary} contains as special case the results in
\cite{LanguascoZ2017e} where $\K = (k, 2, 2)$ and $k \ge 2$.
The case $k_1 = k_2 = k_3 = 3$ has been studied
in~\cite{CantariniGZ2018}, and the more general case
$k_1 = k_2 = \dots = k_s = \ell$ in \cite{Languasco2018}.

\begin{Theorem}
\label{ternary-RH}
Let $\K = (k_1, k_2, k_3)$ where $2 \le k_1 \le k_2 \le k_3$ be a
triple of integers.
For every $\eps > 0$ there exists a constant $C = C(\eps) > 0$,
independent of $\K$, such that
\[
  \sum_{n = N + 1}^{N + H} R(n; \K)
  =
  \frac{\gamma_{k_1} \gamma_{k_2} \gamma_{k_3}}{\Gamma(\rho)}
  H N^{\rho - 1}
  +
  \Odip{\K}{\Phi_{\K}(N ,H)}
\]
as $N \to +\infty$, uniformly for
$H = \infty\bigl( N^{1 - 1 / k_3} (\log N)^6 \bigr)$ with
$H < N^{1 - \eps}$, where
$f = \infty(g)$ means that $g = \odi{f}$ and
$\Phi_{\K}(N, H) = H^2 N^{\rho-2} + H^{1 / 2} N^{\rho - 1/2 - 1 / (2 k_3)} L^3$.
\end{Theorem}

The limitation for $H$ in Theorem~\ref{ternary} is due to the
corresponding one for $\xi$ in Lemma~\ref{LP-Lemma-gen}, while the
limitation for $H$ in Theorem~\ref{ternary-RH} is the expected one.
Theorem~\ref{ternary-RH} for $\K = (3, 3, 3)$ is slightly weaker than
the corresponding result in \cite{CantariniGZ2018}: this is due to the
fact that the identity \eqref{identity} is less efficient than the
special one used there.

We remark that ternary problems are easier to deal with than binary
problems, because we can more efficiently use the H\"older inequality
to bound error terms.
We also remark that we have no constraints on the values of the
exponents $k_1$, $k_2$ and $k_3$, but when they are ``large'' the
range for $H$ reduces correspondingly.

\section{Definitions and preparation for the proofs}

For real $\alpha$ we write $\e(\alpha) = \e^{2 \pi \ii \alpha}$.
We take $N$ as a large positive integer, and write $L = \log N$ for
brevity.
In this and in the following section $k$ denotes any positive real
number.
Let $z = 1 / N - 2 \pi \ii \alpha$ and
\begin{equation}
\label{Stilde-def}
  \Stilde_k(\alpha)
  =
  \sum_{n \ge 1} \Lambda(n) \e^{-n^k / N} \e(n^k \alpha)
  =
  \sum_{n \ge 1} \Lambda(n) \e^{- n^k z}.
\end{equation}
Thus, recalling definition \eqref{Rk-def} and using \eqref{Stilde-def},
for all $n \ge 1$ we have
\begin{equation}
\label{basic-Rk}
  R(n; \K)
  =
  \sum_{n_1^{k_1} + n_2^{k_2} + n_3^{k_3} = n}
    \Lambda(n_1) \Lambda(n_2) \Lambda(n_3)
  =
  \e^{n / N}
  \int_{-1/2}^{1/2}
    \Stilde_{k_1}(\alpha) \Stilde_{k_2}(\alpha) \Stilde_{k_3}(\alpha)
    \, \e(-n \alpha) \, \dx \alpha.
\end{equation}
It is clear from the above identity that we are only interested in the
range $\alpha \in [-1/2, 1/2]$.
We record here the basic inequality
\begin{equation}
\label{z-bound}
  \vert z \vert^{-1}
  \ll
  \min \{ N, \vert \alpha \vert^{-1} \}.
\end{equation}
We also need the following exponential sum over the ``short interval''
$[1, H]$
\[
  U(\alpha, H)
  =
  \sum_{m = 1}^H \e(m \alpha),
\]
where $1 \le H \le N$ is a large integer.
We recall the simple inequality
\begin{equation}
\label{U-bound}
  \vert U(\alpha, H) \vert
  \le
  \min \{ H, \vert \alpha \vert^{-1} \}.
\end{equation}
With these definitions in mind and recalling \eqref{basic-Rk}, we
remark that
\begin{equation}
\label{basic-identity}
  \sum_{n = N + 1}^{N + H}
    \e^{-n / N} R(n; \K)
  =
  \int_{-1/2}^{1/2}
    \Stilde_{k_1}(\alpha) \Stilde_{k_2}(\alpha) \Stilde_{k_3}(\alpha)
    U(-\alpha, H) \, \e(-N \alpha) \, \dx \alpha,
\end{equation}
which is the starting point for our investigation.
The basic strategy is to replace each factor $\Stilde_k(\alpha)$ by
its expected main term, which is $\gamma_k / z^{1 / k}$, and
estimating the ensuing error term by means of a combination of
techniques and bounds for exponential sums.
One key ingredient is the $L^2$-bound in Lemma~\ref{LP-Lemma-gen},
which we may use only in a restricted range, and we need a different
argument on the remaining part of the integration interval; this leads
to some complications in details in the proof of the unconditional
result.

\section{Lemmas}

For brevity, we set
\[
  \E_k(\alpha)
  :=
  \Stilde_k(\alpha)
  -
  \frac{\gamma_k}{z^{1 / k}}
  \qquad\text{and}\qquad
  A(N; c)
  :=
  \exp\Bigl\{ c \Bigl(\frac{\log N}{\log\log N}\Bigr)^{1/3} \Bigr\},
\]
where $c$ is a real constant.

\begin{Lemma}[Lemma~3 of \cite{LanguascoZ2016a}]
\label{LP-Lemma-gen}
Let $\eps$ be an arbitrarily small positive constant, $k \ge 1$ be
an integer, $N$ be a sufficiently large integer and $L = \log N$.
Then there exists a positive constant $c_1 = c_1(\eps)$, which does
not depend on $k$, such that
\[
  \int_{-\xi}^{\xi}
    \bigl\vert \E_k(\alpha) \bigr\vert^2 \, \dx \alpha
  \ll_k
  N^{2 / k - 1} A(N; - c_1)
\]
uniformly for $0 \le \xi < N^{ -1 + 5 / (6 k) - \eps}$.
Assuming the Riemann Hypothesis we have
\[
  \int_{-\xi}^{\xi} \,
    \bigl\vert \E_k(\alpha) \bigr\vert^2 \, \dx \alpha
  \ll_k
  N^{1 / k}\xi L^2
\]
uniformly for $0 \le \xi \le 1 / 2$.
\end{Lemma}

We remark that the proof of Lemma~3 in \cite{LanguascoZ2016a} contains
oversights which are corrected in \cite{LanguascoZ2017e}.
The next result is a variant of Lemma~4 of \cite{LanguascoZ2016a}: we
just follow the proof until the last step.
We need it to avoid dealing with the ``periphery'' of the major arc in
the unconditional case.

\begin{Lemma}[Lemma~4 of \cite{LanguascoZ2016a}]
\label{Laplace-formula}
Let $N$ be a positive integer, $z = z(\alpha) = 1 / N - 2 \pi \ii \alpha$,
and $\mu > 0$.
Then, uniformly for $n \ge 1$ and $X > 0$ we have
\[
  \int_{-X}^X z^{-\mu} \e(-n \alpha) \, \dx \alpha
  =
  \e^{- n / N} \frac{n^{\mu - 1}}{\Gamma(\mu)}
  +
  \Odip{\mu}{\frac1{n X^{\mu}}}.
\]
\end{Lemma}

\begin{Lemma}[Lemma 3.3 of \cite{CantariniGZ2018}]
\label{Stilde-bound}
We have $\Stilde_k(\alpha) \ll_k N^{1 / k}$.
\end{Lemma}

This is a consequence of the Prime Number Theorem.
We notice that by Lemma~\ref{Stilde-bound} and~\eqref{z-bound} we have
\begin{equation}
\label{y-bound}
  \E_k(\alpha)
  =
  \Stilde_k(\alpha) - \frac{\gamma_k}{z^{1 / k}}
  \ll_k
  N^{1 / k}.
\end{equation}

Our next tool is the extension to $\Stilde_k$ of Lemma~7 of Tolev
\cite{Tolev1992}.
A simple integration by parts then yields Lemma~\ref{weighted-L2}.

\begin{Lemma}
\label{Tolev-Lemma}
Let $k > 1$ and $\tau > 0$. Then
\[
  \int_{-\tau}^{\tau} \vert \Stilde_k(\alpha) \vert^2 \, \dx \alpha
  \ll_k
  \bigl(\tau N^{1/k} + N^{2/k - 1}\bigr) L^3.
\]
\end{Lemma}

\begin{proof}
Letting $P=(2NL/k)^{1/k}$, a direct estimate
gives $\Stilde_{k}(\alpha)= \sum_{n\le P} \Lambda(n) \e^{-n^k/N} \e(n^k\alpha) + \Odipm{k}{L^{1/k}}$.
Recalling that the Prime Number Theorem implies
$S_{k}(\alpha;t) := \sum_{n\le t} \Lambda(n) \e(n^{k}\alpha)  \ll t$, a
partial integration argument gives
\[
\sum_{n\le P} \Lambda(n) \e^{-n^k/N} \e(n^k\alpha)
=
-\frac{k}{N} \int_1^P t^{k-1} \e^{-t^k/N} S_{k}(\alpha;t) \  \dx t
+ \Odipm{k}{L^{1/k}}.
\]
Using the inequality $(\vert a\vert + \vert b \vert)^{2} \ll \vert a\vert^{2} + \vert b \vert^{2}$, Cauchy-Schwarz inequality
and interchanging the integrals, we get that
\begin{align*}
\int_{-\tau}^{\tau}
&\vert \Stilde_{k}(\alpha)\vert ^{2} \ \dx\alpha
\ll_{k}
\int_{-\tau}^{\tau}
\Bigl\vert \frac{1}{N} \int_1^P t^{k-1} \e^{-t^k/N} S_{k}(\alpha;t)\  \dx t \Bigr\vert ^{2} \ \dx\alpha
+  L^{2/k}
\\
&\ll_{k}
\frac{1}{N^{2}}
\Bigl( \int_1^P t^{k-1} \e^{-t^k/N} \  \dx t \Bigr)
\Bigl( \int_1^P t^{k-1} \e^{-t^k/N} \int_{-\tau}^{\tau}\vert S_{k}(\alpha;t)\vert ^{2} \dx\alpha  \  \dx t \Bigr)
+  L^{2/k}.
\end{align*}
Lemma 7 of Tolev \cite{Tolev1992} in the form given in Lemma 5 of \cite{GambiniLZ2018} on
$S_{k}(\alpha;t) = \sum_{n\le t} \Lambda(n) \e(n^k\alpha)$ implies
that
$\int_{-\tau}^{\tau} \vert S_{k}(\alpha;t)\vert^{2}\ \dx\alpha \ll_{k}
\bigl(\tau\ t+t^{2-k}\bigr)(\log t)^3 $.
Using such an estimate
and remarking that $ \int_1^P t^{k-1} \e^{-t^k/N} \  \dx t \ll_k N$,
we obtain  that
\begin{align*}
\int_{-\tau}^{\tau}
\vert \Stilde_{k}(\alpha)\vert ^{2} \ \dx\alpha
& \ll_{k}
\frac{1}{N}
 \int_1^P  \bigl(\tau\ t+t^{2-k}\bigr) t^{k-1} \e^{-t^k/N} (\log t)^3 \  \dx t
+  L^{2/k}
\\&
 \ll_{k}
 \bigl(\tau N^{1/k}+N^{2/k-1}\bigr)L^{3}
\end{align*}
by a direct computation.
\end{proof}

\begin{Lemma}
\label{weighted-L2}
For $k > 1$ and $N^{-c} \le \tau \le 1 / 2$, where $c > 0$ is
fixed, we have
\[
  \int_\tau^{1 / 2}
    \vert \Stilde_k(\alpha) \vert^2 \, \frac{\dx \alpha}\alpha
  \ll_k
  N^{1 / k} L^4 + \tau^{-1} N^{2 / k - 1} L^3.
\]
\end{Lemma}

\begin{Lemma}[Lemma 3.6 of \cite{CantariniGZ2018}]
\label{mt-evaluation}
For $N \to +\infty$, $H \in [1, N]$ and a real number $\lambda$ we
have
\[
  \sum_{n = N + 1}^{N + H}
    \e^{- n / N} n^{\lambda}
  =
  \frac1{\e}
  H N^{\lambda}
  +
  \Odipm{\lambda}{H^2 N^{\lambda - 1}}.
\]
\end{Lemma}

\section{Proof of Theorem \ref{ternary}}

We recall that $\K = (k_1, k_2, k_3)$ where $k_j \ge 2$ is an integer
and that $\rho = 1 / k_1 + 1 / k_2 + 1 / k_3$ is the density of our
problem.
We need to introduce another parameter $B = B(N)$, defined as
\begin{equation}
\label{def-B}
  B = N^{2 \eps},
\end{equation}
where $\eps > 0$ is fixed.
Ideally, we would like to take $B = 1$, but we are prevented from
doing this by the estimate in \S\ref{sub-I5}.
We let $\C = \C(B,H) = [-1/2, -B/H] \cup [B/H, 1/2]$.
We write $\Stilde_{k_j}(\alpha) = x_j + y_j$ where
$x_j = x_j(\alpha) = \gamma_j z^{-1/k_j}$ and
$y_j = y_j(\alpha) = \E_{k_j}(\alpha)$, so that
\begin{equation}
\label{identity}
  \Stilde_{k_1}(\alpha) \Stilde_{k_2}(\alpha) \Stilde_{k_3}(\alpha)
  =
  (x_1 + y_1) (x_2 + y_2) (x_3 + y_3)
  =
  x_1 x_2 x_3 + \gA - \gB - \gC,
\end{equation}
where
$\gA(\alpha) = y_1 \Stilde_{k_2}(\alpha) \Stilde_{k_3}(\alpha) +
             \Stilde_{k_1}(\alpha) y_2 \Stilde_{k_3}(\alpha) +
             \Stilde_{k_1}(\alpha) \Stilde_{k_2}(\alpha) y_3$,
$\gB(\alpha) = x_1 y_2 y_3 + y_1 x_2 y_3 + y_1 y_2 x_3$ and
$\gC(\alpha) = 2 y_1 y_2 y_3$.
We multiply \eqref{identity} by $U(-\alpha, H) \* \e(-N \alpha)$
and integrate over the interval $[-B / H, B / H]$.
Recalling \eqref{basic-identity} we have
\begin{align*}
  \sum_{n = N + 1}^{N + H}
    \e^{-n / N} R(n; \K)
  &=
  \gamma_{k_1} \gamma_{k_2} \gamma_{k_3}
  \int_{-B/H}^{B/H} \frac{U(-\alpha, H)}{z^{\rho}}
    \, \e(-N \alpha) \, \dx \alpha \\
  &\qquad+
  \int_{-B/H}^{B/H} \gA(\alpha) U(-\alpha, H) \e(-N \alpha) \, \dx \alpha \\
  &\qquad-
  \int_{-B/H}^{B/H} \gB(\alpha) U(-\alpha, H) \e(-N \alpha) \, \dx \alpha \\
  &\qquad-
  \int_{-B/H}^{B/H} \gC(\alpha) U(-\alpha, H) \e(-N \alpha) \, \dx \alpha \\
  &\qquad+
  \int_\C \Stilde_{k_1}(\alpha) \Stilde_{k_2}(\alpha)\Stilde_{k_3}(\alpha)
    U(-\alpha, H) \e(-N \alpha) \, \dx \alpha \\
  &=
  \gamma_{k_1} \gamma_{k_2} \gamma_{k_3} I_1
  +
  I_2 - I_3 - I_4 + I_5,
\end{align*}
say.
The first summand gives rise to the main term via
Lemma~\ref{Laplace-formula}, the next three are majorised in
\S\ref{sub-I2}--\ref{sub-I5} by means of Lemma~\ref{Stilde-bound} and
the $L^2$-estimate provided by Lemma~\ref{LP-Lemma-gen}.
Finally, $I_5$ is easy to bound using Lemma~\ref{weighted-L2}.

\subsection{Evaluation of \texorpdfstring{$I_1$}{I1}}

It is a straightforward application of Lemma~\ref{Laplace-formula}:
here we exploit the flexibility of having variable endpoints instead
of the full unit interval.
We have
\begin{equation}
\label{prep-mt}
  \int_{-B/H}^{B/H} \frac{U(-\alpha, H)}{z^{\rho}}
    \, \e(-N \alpha) \, \dx \alpha
  =
  \frac1{\Gamma(\rho)}
  \sum_{n = N + 1}^{N + H}
    \e^{- n / N} n^{\rho - 1}
  +
  \Odip{\K}{\frac HN \Bigl( \frac HB \Bigr)^{\rho}}.
\end{equation}
We evaluate the sum on the right-hand side of \eqref{prep-mt} by means
of Lemma~\ref{mt-evaluation} with $\lambda = \rho - 1$.
Summing up, we have
\begin{equation}
\label{final-mt}
  \int_{-B/H}^{B/H} \frac{U(-\alpha, H)}{z^{\rho}}
    \, \e(-N \alpha) \, \dx \alpha
  =
  \frac1{\e \Gamma(\rho)} H N^{\rho - 1}
  +
  \Odip{\K}{H^2 N^{\rho - 2} + \frac HN \Bigl( \frac HB \Bigr)^{\rho}}.
\end{equation}

It is now convenient to choose the range for $H$: keeping in mind that
will need Lemma~\ref{LP-Lemma-gen}, we see that we can take
\begin{equation}
\label{H-bound}
  H > N^{1 - 5 / (6 \max k_j) + 3 \eps}.
\end{equation}

\subsection{Bound for \texorpdfstring{$I_2$}{I2}}
\label{sub-I2}

We recall the bound~\eqref{U-bound}, and Lemmas~\ref{Stilde-bound} and
\ref{Tolev-Lemma}.
Using Lemma~\ref{LP-Lemma-gen} and the Cauchy-Schwarz inequality where
appropriate, we see that the contribution from
$\Stilde_{k_1}(\alpha) \* \Stilde_{k_2}(\alpha) y_3$, say, is
\begin{align}
\notag
  &\ll_{\K}
  H
  \max_{\alpha \in [-1/2, 1/2]} \vert \Stilde_{k_1}(\alpha) \vert
  \Bigl(
  \int_{-B / H}^{B / H} \vert \Stilde_{k_2}(\alpha) \vert^2 \, \dx \alpha
  \int_{-B / H}^{B / H} \vert \E_{k_3}(\alpha) \vert^2 \, \dx \alpha
  \Bigr)^{1 / 2} \\
\notag
  &\ll_{\K}
  H N^{1 / k_1} L^{3 / 2}
  \Bigl( \frac BH N^{1 / k_2} + N^{2 / k_2 - 1} \Bigr)^{1 / 2}
  \bigl( N^{2 / k_3 - 1} A(N; - c_1) \bigr)^{1 / 2} \\
\label{bound-I2}
  &\ll_{\K}
  H N^{\rho - 1} A \Bigl(N; - \frac13 c_1\Bigr),
\end{align}
where $c_1 = c_1(\eps) > 0$ is the constant provided by
Lemma~\ref{LP-Lemma-gen}, which we can use on the interval
$[-B / H, B / H]$ since $B$ and $H$ satisfy \eqref{def-B} and
\eqref{H-bound} respectively.
The other two summands in $I_2$ are treated in the same way.

\subsection{Bounds for \texorpdfstring{$I_3$}{I3} and \texorpdfstring{$I_4$}{I4}}

Using \eqref{z-bound}, \eqref{U-bound} and Lemma~\ref{LP-Lemma-gen},
by the Cauchy-Schwarz inequality, we see that the contribution
from the term $y_1 y_2 x_3$ is
\begin{align}
\notag
  &=
  \gamma_{k_3}
  \int_{-B/H}^{B/H}
    \frac{\E_{k_1}(\alpha) \E_{k_2}(\alpha)}{z^{1/k_3}}
    U(-\alpha, H) \e(-N \alpha) \, \dx \alpha \\
\notag
  &\ll_{\K}
  H N^{1 / k_3}
  \Bigl(
    \int_{-B/H}^{B/H} \vert \E_{k_1}(\alpha) \vert^2 \, \dx \alpha
    \int_{-B/H}^{B/H} \vert \E_{k_2}(\alpha) \vert^2 \, \dx \alpha
  \Bigr)^{1/2} \\
\label{bound-I3}
  &\ll_{\K}
  H N^{\rho - 1} A(N; -c_1).
\end{align}
The other two summands in $I_3$ are treated in the same way.
Furthermore, we notice that $y_3 \ll_{k_3} N^{1 / k_3}$ by
\eqref{y-bound}, and the contribution from $\gC(\alpha)$ is also
bounded as in \eqref{bound-I3}.

\subsection{Bound for \texorpdfstring{$I_5$}{I5}}
\label{sub-I5}

Using \eqref{U-bound}, Lemma~\ref{weighted-L2} and the Cauchy-Schwarz
inequality, we have
\begin{align}
\notag
  I_5
  &=
  \int_\C \Stilde_{k_1}(\alpha) \Stilde_{k_2}(\alpha) \Stilde_{k_3}(\alpha)
    U(-\alpha, H) \e(-N \alpha) \, \dx \alpha \\
\notag
  &\ll_{\K}
  \max_{\alpha \in [-1/2, 1/2]} \vert \Stilde_{k_1}(\alpha) \vert
  \Bigl(
    \int_\C \vert \Stilde_{k_2}(\alpha) \vert^2
      \, \frac{\dx \alpha}{\vert \alpha \vert}
    \int_\C \vert \Stilde_{k_3}(\alpha) \vert^2
      \, \frac{\dx \alpha}{\vert \alpha \vert}
  \Bigr)^{1/2} \\
\label{bound-I5}
  &\ll_{\K}
  N^{1 / k_1}
  \Bigl( \frac{H^2}{B^2} N^{2 / k_2 + 2 / k_3 - 2} L^6 \Bigr)^{1/2}
  \ll_{\K}
  \frac HB N^{\rho - 1} L^3,
\end{align}
because of \eqref{H-bound}.
This is $\ll_{\K} H N^{\rho - 1} A(N; -c_1 / 3)$, by our choice
in~\eqref{def-B}.

\subsection{Completion of the proof}
\label{final-Th1}

For simplicity, from now on we assume that $H \le N^{1 - \eps}$.
Summing up from \eqref{final-mt}, \eqref{bound-I2}, \eqref{bound-I3}
and \eqref{bound-I5}, we proved that
\begin{equation}
\label{smooth-Th1}
  \sum_{n = N + 1}^{N + H}
    \e^{-n / N} R(n; \K)
  =
  \frac{\gamma_{k_1} \gamma_{k_2} \gamma_{k_3}}{\e \Gamma(\rho)}
  H N^{\rho - 1}
  +
  \Odip{\K}{H N^{\rho - 1} A \Bigl(N; - \frac13 c_1\Bigr)},
\end{equation}
provided that \eqref{def-B} and \eqref{H-bound} hold, since the other
error terms are smaller in our range for $H$.
In order to achieve the proof, we have to remove the exponential
factor on the left-hand side, exploiting the fact that, since $H$ is
``small,'' it does not vary too much over the summation range.
Since $\e^{-n / N} \in [\e^{-2}, \e^{-1}]$ for all
$n \in [N + 1, N + H]$,
we can easily deduce from \eqref{smooth-Th1} that
\[
  \e^{-2}
  \sum_{n = N + 1}^{N + H} R(n; \K)
  \le
  \sum_{n = N + 1}^{N + H}
    \e^{-n / N} R(n; \K)
  \ll_{\K}
  H N^{\rho - 1}.
\]
We can use this weak upper bound to majorise the error term arising
from the development $\e^{-x} = 1 + \Odi{x}$ that we need in the
left-hand side of \eqref{smooth-Th1}.
In fact, we have
\begin{align*}
  \sum_{n = N + 1}^{N + H}
    \e^{-n / N} R(n; \K)
  &=
  \sum_{n = N + 1}^{N + H}
    \bigl(\e^{-1} + \Odi{(n - N) N^{-1}} \bigr) R(n; \K) \\
  &=
  \e^{-1}
  \sum_{n = N + 1}^{N + H} R(n; \K)
  +
  \Odip{\K}{ H^2 N^{\rho - 2} }.
\end{align*}
Finally, substituting back into \eqref{smooth-Th1}, we obtain the
required asymptotic formula for $H$ as in the statement of
Theorem~\ref{ternary}.

\section{Proof of Theorem \ref{ternary-RH}}

In the conditional case, we can use identity~\eqref{identity} over the
whole interval $[-1/2, 1/2]$.
Recalling \eqref{basic-identity} we have
\begin{align*}
  \sum_{n = N + 1}^{N + H}
    \e^{-n / N} R(n; \K)
  &=
  \gamma_{k_1} \gamma_{k_2} \gamma_{k_3}
  \int_{-1 / 2}^{1 / 2} \frac{U(-\alpha, H)}{z^{\rho}}
    \, \e(-N \alpha) \, \dx \alpha \\
  &\qquad+
  \int_{-1 / 2}^{1 / 2} \gA(\alpha) U(-\alpha, H) \e(-N \alpha) \, \dx \alpha \\
  &\qquad-
  \int_{-1 / 2}^{1 / 2} \gB(\alpha) U(-\alpha, H) \e(-N \alpha) \, \dx \alpha \\
  &\qquad-
  \int_{-1 / 2}^{1 / 2} \gC(\alpha) U(-\alpha, H) \e(-N \alpha) \, \dx \alpha \\
  &=
  \gamma_{k_1} \gamma_{k_2} \gamma_{k_3} I_1
  +
  I_2 - I_3 - I_4,
\end{align*}
say.
For the main term we use Lemma~\ref{Laplace-formula} over $[-1/2, 1/2]$
and then Lemma~\ref{mt-evaluation} with $\lambda = \rho - 1$, obtaining
\begin{equation}
\label{final-mt-rh}
  \int_{-1 / 2}^{1 / 2} \frac{U(-\alpha, H)}{z^{\rho}}
    \, \e(-N \alpha) \, \dx \alpha
  =
  \frac1{\e \Gamma(\rho)} H N^{\rho - 1}
  +
  \Odip{\K}{H^2 N^{\rho - 2} + \frac HN}.
\end{equation}
For the other terms, we split the integration range at $1 / H$.
We use Lemma~\ref{LP-Lemma-gen} and \eqref{U-bound} on the interval
$[-1 / H, 1 / H]$,  and a partial-integration argument
from Lemma~\ref{LP-Lemma-gen} in the remaining range.

In view of future constraints (see \eqref{bound-H-RH} below) we assume
that
\begin{equation}
\label{first-bound-H}
  H
  \ge
  N^{1 - 1 / k_3} L.
\end{equation}
We start bounding the contribution of the term
$\Stilde_{k_1}(\alpha) \* \Stilde_{k_2}(\alpha) y_3$ in $\gA(\alpha)$
over $[-1 / H, 1 / H]$.
We have that it is
\begin{align*}
  &\ll_{\K}
  H
  \max_{\alpha \in [-1/2, 1/2]} \vert \Stilde_{k_1}(\alpha) \vert
  \Bigl(
  \int_{-1 / H}^{1 / H} \vert \Stilde_{k_2}(\alpha) \vert^2 \, \dx \alpha
  \int_{-1 / H}^{1 / H} \vert \E_{k_3}(\alpha) \vert^2 \, \dx \alpha
  \Bigr)^{1 / 2} \\
  &\ll_{\K}
  H N^{1 / k_1} L^{3 / 2}
  \Bigl( \frac 1H N^{1 / k_2} + N^{2 / k_2 - 1} \Bigr)^{1 / 2}
  \bigl( N^{1 / k_3} H^{-1} L^2 \bigr)^{1 / 2} \\
  &\ll_{\K}
  H^{1 / 2} N^{\rho - 1/2 - 1 / (2 k_3)} L^{5/2},
\end{align*}
by Lemma~\ref{LP-Lemma-gen}, since we assumed \eqref{first-bound-H}.
The other two summands in $I_2$ are treated in the same way.
Next, we bound the contribution of the term $x_1 y_2 y_3$ in
$\gB(\alpha)$ on the same interval: it is
\begin{align}
\notag
  &=
  \gamma_{k_1}
  \int_{-1/H}^{1/H}
    \frac{\E_{k_2}(\alpha) \E_{k_3}(\alpha)}{z^{1/k_1}}
    U(-\alpha, H) \e(-N \alpha) \, \dx \alpha \\
\notag
  &\ll_{\K}
  H N^{1 / k_1}
  \Bigl(
    \int_{-1/H}^{1/H} \vert \E_{k_2}(\alpha) \vert^2 \, \dx \alpha
    \int_{-1/H}^{1/H} \vert \E_{k_3}(\alpha) \vert^2 \, \dx \alpha
  \Bigr)^{1/2} \\
\notag
  &\ll_{\K}
  H N^{1 / k_1}
  \Bigl( N^{1 / k_2 + 1 / k_3} \frac 1{H^2} L^4 \Bigr)^{1/2}
  \ll_{\K}
  N^{\rho - 1 / (2 k_2) - 1 / (2 k_3)} L^2.
\end{align}
The other two summands in $I_3$ are treated in the same way.
Furthermore, we recall that $\E_k(\alpha) \ll_k N^{1 / k}$ by
\eqref{y-bound}, and the contribution from $\gC(\alpha)$ can also be
bounded as above.

We now deal with  the remaining range $\C = [-1/2,1/2] \setminus [-1/H,1/H]$.
Arguing as in (16) of \cite{CantariniGZ2018} by partial integration
from Lemma~\ref{LP-Lemma-gen}, for $k > 1$ we have
\[
  \int_{\C}
    \bigl\vert \E_k(\alpha) \bigr\vert^2
    \, \frac{\dx \alpha}{\vert \alpha \vert}
  \ll_k
  N^{1 / k} L^3.
\]
Proceeding as above, we start bounding the contribution of the term
$\Stilde_{k_1}(\alpha) \* \Stilde_{k_2}(\alpha) y_3$ in $\gA(\alpha)$.
Using \eqref{U-bound} and Lemma~\ref{weighted-L2} we see that it is
\begin{align*}
  &\ll_{\K}
  \max_{\alpha \in [-1/2, 1/2]} \vert \Stilde_{k_1}(\alpha) \vert
  \Bigl(
  \int_{\C}
    \vert \Stilde_{k_2}(\alpha) \vert^2 \, \frac{\dx \alpha}{\vert \alpha \vert}
  \int_{\C}
    \vert \E_{k_3}(\alpha) \vert^2 \, \frac{\dx \alpha}{\vert \alpha \vert}
  \Bigr)^{1 / 2} \\
  &\ll_{\K}
  N^{1 / k_1}
  \Bigl( N^{1 / k_2} L^4 + H N^{(2 - k_2) / k_2} L^3 \Bigr)^{1 / 2}
  \bigl( N^{1 / k_3} L^3 \bigr)^{1 / 2} \\
  &\ll_{\K}
  H^{1 / 2} N^{\rho - 1/2 - 1 / (2 k_3)} L^3,
\end{align*}
since we assumed \eqref{first-bound-H}.
The other two summands in $I_2$ are treated in the same way.
Next, we bound the contribution of the term $x_1 y_2 y_3$ in
$\gB(\alpha)$ on the same interval: using \eqref{U-bound} again, it is
\begin{align}
\notag
  &\ll_{\K}
  \int_{\C}
    \frac{\E_{k_2}(\alpha) \E_{k_3}(\alpha)}{\vert z \vert^{1 / k_1}}
    \, \frac{\dx \alpha}{\vert \alpha \vert} \\
\notag
  &\ll_{\K}
  N^{1 / k_1}
  \Bigl(
    \int_{\C}
      \vert \E_{k_2}(\alpha) \vert^2 \, \frac{\dx \alpha}{\vert \alpha \vert}
    \int_{\C}
      \vert \E_{k_3}(\alpha) \vert^2 \, \frac{\dx \alpha}{\vert \alpha \vert}
  \Bigr)^{1/2} \\
\notag
  &\ll_{\K}
  N^{1 / k_1}
  \bigl(
    N^{1 / k_2 + 1 / k_3} L^6
  \bigr)^{1/2}
  \ll_{\K}
  N^{\rho - 1 / (2 k_2) - 1 / (2 k_3)} L^3.
\end{align}
The other two summands in $I_3$ are treated in the same way.
The contribution from $\gC(\alpha)$ can also be bounded as above.

Summing up from~\eqref{final-mt-rh}, recalling that
$2 \le k_1 \le k_2 \le k_3$, we proved that
\[
  \sum_{n = N + 1}^{N + H}
    \e^{-n / N} R(n; \K)
  =
  \frac1{\e \Gamma(\rho)} H N^{\rho - 1}
  +
  \Odip{\K}{\Psi_{\K}(N, H)},
\]
where
\[
  \Psi_{\K}(N, H)
  =
  H^2 N^{\rho - 2}
  +
  H^{1 / 2} N^{\rho - 1/2 - 1 / (2 k_3)} L^3
  +
  N^{\rho - 1 / (2 k_2) - 1 / (2 k_3)} L^3.
\]
We dropped the term $H N^{-1}$ which is smaller than $H^2 N^{\rho - 2}$
because of \eqref{first-bound-H}.
Since we want an asymptotic formula, we need to impose the restriction
\begin{equation}
\label{bound-H-RH}
  H = \infty \bigl( N^{1 - 1 / k_3} L^6 \bigr),
\end{equation}
which supersedes \eqref{first-bound-H}.
Therefore, we may take
\begin{equation}
\label{def-Phi}
  \Phi_{\K}(N, H)
  =
  H^2 N^{\rho - 2}
  +
  H^{1 / 2} N^{\rho - 1/2 - 1 / (2 k_3)} L^3.
\end{equation}
We remark that when $k_1 = 2$ we can use Lemma~2
of~\cite{LanguascoZ2016c} instead of Lemma~\ref{Tolev-Lemma} in the
partial integration in the proof of Lemma~\ref{weighted-L2}, and we
can replace the right-hand side by $N^{1 / 2} L^2 + H L^2$.
This means, in particular, that, in this case, we may replace $L^3$
in the far right of \eqref{def-Phi} by $L^{5 / 2}$.

Next, we remove the exponential weight, arguing essentially as in
\S\ref{final-Th1}.
This completes the proof of Theorem~\ref{ternary-RH}.

\subsubsection*{Acknowledgment}
The first Author gratefully acknowledges support from a grant
``Ing.~Giorgio Schirillo'' from Istituto Nazionale di Alta Matematica.

\providecommand{\MR}{\relax\ifhmode\unskip\space\fi MR }
\providecommand{\MRhref}[2]{%
  \href{http://www.ams.org/mathscinet-getitem?mr=#1}{#2}
}
\providecommand{\href}[2]{#2}

\bigskip

\begin{tabular}{ll}
MC, AG, AZ                            & AL \\
Dipartimento di Scienze, Matematiche, & Dipartimento di Matematica \\
\quad Fisiche e Informatiche          & \quad ``Tullio Levi-Civita'' \\
Universit\`a di Parma                 & Universit\`a di Padova \\
Parco Area delle Scienze 53/a         & Via Trieste 63 \\
43124 Parma, Italia                   & 35121 Padova, Italia \\
\\
email (MC): \texttt{cantarini\_m@libero.it} \\
email (AG): \texttt{a.gambini@unibo.it} \\
email (AL): \texttt{alessandro.languasco@unipd.it} \\
email (AZ): \texttt{alessandro.zaccagnini@unipr.it}
\end{tabular}

\end{document}